\UseRawInputEncoding
\documentclass[12pt]{article}
\usepackage{float}
\usepackage{color}
\usepackage{subfigure}
\usepackage{amsmath}
\usepackage{amsthm}
\usepackage{amsfonts}
\usepackage{amssymb,latexsym}
\usepackage[all]{xy}
\usepackage{graphicx}
\usepackage[mathscr]{eucal}
\usepackage{verbatim}
\usepackage{hyperref}

\theoremstyle{plain}

    \newtheorem{thm}{Theorem}[section]
       
       \newtheorem{lem}{Lemma}[section]
       \newtheorem{cor}{Corollary}[section]
       \newtheorem{defn}{Definition}[section]
       \newtheorem{rem}{Remark}[section]

\numberwithin{equation}{section}

\usepackage{graphicx}

\begin{document}
\title{Exponential attractors with explicit fractal dimensions for retarded functional differential equations}

\author{Wenjie Hu$^{1,2}$,  Tom\'{a}s
Caraballo$^{3}$\footnote{Corresponding author.  E-mail address: caraball@us.es (Tom\'{a}s
Caraballo).}
\\
\small  1. The MOE-LCSM, School of Mathematics and Statistics,  Hunan Normal University,\\
\small Changsha, Hunan 410081, China\\
\small  2. Journal House, Hunan Normal University, Changsha, Hunan 410081, China\\
\small 3 Dpto. Ecuaciones Diferenciales y An\'{a}lisis Num\'{e}rico, Facultad de Matem\'{a}ticas,\\
\small  Universidad de Sevilla, c/ Tarfia s/n, 41012-Sevilla, Spain
}

\date {}
\maketitle

\begin{abstract}
The aim of this paper is to propose a new  method to construct exponential attractors for infinite dimensional dynamical systems in Banach spaces with explicit fractal dimension. The approach is established by combing the squeezing properties and the covering of finite subspace of Banach spaces,  which generalize the method established in Hilbert spaces. The constructed  exponential attractors possess explicit  fractal dimensions which do not depend on the entropy number but only depend on the spectrum of the linear part and Lipschitz constant of the nonlinear part. The method is especially effective for functional differential equations in Banach spaces for which sate decomposition of the linear part can be adopted to prove squeezing property. The theoretical results are applied to  the retarded functional differential equation and retarded reaction-diffusion equations.
\end{abstract}

\bigskip


\section{Introduction}
Exponential attractors of infinite dimensional dynamical systems are compact subsets  of the phase space with finite fractal dimensions, which are positively invariant and attract all bounded subsets at  exponential rates. It is well known that if exponential attractors  exist, then they contain global attractors. Although they may be larger than the global attractors, they are more robust  than global attractors under perturbations due to the exponential rates of convergence. Hence, they paly significant roles in investigating asymptotic behavior of infinite dimensional nonlinear dynamical systems especially for those with fast convergence rate. Eden, Foias, Nicolaenco and Temam \cite{EFNT94} first proposed the concept of exponential attractor, where the theory was established based on the squeezing property in the Hilbert spaces.

However, there are many evolution equations arising from real world modelings defined in Banach spaces, such as the  delayed differential equations \cite{JH} and delayed partial differential  equations \cite{WJ}. Therefore, one natural question arises, how to construct exponential attractors for infinite dimensional systems in Banach spaces?  Efendiev, Miranville  Zelik \cite{C17,C18} adopted a so-called smoothing property of the semigroup between two different Banach spaces to construct existence of  exponential and uniform attractors for systems in Banach spaces, which has also been widely used in estimating the fractal dimension and construct exponential attractors of various evolutions equations, see for instance, \cite{C9,C17,C20,C21,NHC21,LHSS20} and the references therein.

It should be pointed out that in the  recent work recent work \cite{C9}, the authors extended the results established in  \cite{C17,C18} to the case of time continuous asymptotically compact evolution processes in Banach spaces, which was adopted to investigate the pullback exponential attractors of nonautonomous retarded functional differential equations in the very recent works \cite{NHC21,LHSS20} and non-autonomous Chafee-Infante equation as well as non-autonomous dissipative wave equation in \cite{5}. Nevertheless, the construction in these works generally can not give explicit bound of  the fractal dimensions since it depends on the choice of another embedding space which may variety from space to space. Furthermore, the dimension estimation depends on the entropy number between two spaces which is in general quite difficult to obtain an explicit bound. Indeed, in \cite{LHSS20}, the authors pointed out  that only for scalar equations the entropy numbers of the embedding $C \hookrightarrow C^1$ are explicitly known, which yields an estimate for the fractal dimension of the exponential attractors. Hence, one naturally wonders whether we can construct  exponential attractors with explicit bounds of fractal dimensions for systems in Banach spaces that only depend on the inner characteristic of the system? In this work, we affirm this by extending Eden, Foias, Nicolaenco and Temam's work \cite{EFNT94} work to Banach spaces.

The outline of our paper is as follows. In Section 2, we recall basic notions and results from the theory of infinite dimensional dynamical systems and propose the construction procedure of exponential attractors with explicit fractal dimensions for infinite dimensional dynamical systems in Banach spaces. In Section 3, we illustrate effectiveness of our obtained theoretical results by applications to the system  generated by autonomous retarded functional  differential equations and retarded reaction-diffusion equations. At last, we summarize the paper by pointing out some potential directions for future studies.
\section{Exponential attractors}
We first introduce some preliminaries for establishing our main results, including the definitions of evolution process, exponential attractors as well as some hypothesis.
\begin{defn}\label{defn4.1}
Let $X$ be a Banch space, $\mathbb{R}^+=[0,\infty)$. A family of mappings $S(t): X \rightarrow X$ is said to be a semigroup, provided
 \begin{equation}\label{2.1}
\begin{aligned}
S(t+s) &=S(t)S(s), \quad \forall t,  s \in \mathbb{R}^+, \\
S(0) & =\mathrm{Id}_X,
\end{aligned}
\end{equation}
where $\mathrm{Id}_X: X \rightarrow X$ represents the identity map on $X$.
\end{defn}

We now give the following definition of exponential attractors.
\begin{defn}\label{defn2.1}Let $S(t)$ be a semigroup in $X$. A  non-empty compact subset $\mathcal{M}$ is called an
exponential attractor of the semigroup $S(t)$ if\\
(i) $\mathcal{M}$ is positively invariant, i.e.
$$
S(t)\mathcal{M} \subset \mathcal{M}  \quad \forall t \in \mathbb{R}^+.
$$\\
(ii) the fractal dimension $\operatorname{dim}_f(\mathcal{M})$ of $\mathcal{M}$ is bounded,
where $\operatorname{dim}_f(\mathcal{M})$ is  defined as
$$
\operatorname{dim}_f(\mathcal{M})=\lim _{\varepsilon \rightarrow 0} \frac{\ln \left(N_{\varepsilon}^X(\mathcal{M})\right)}{\ln \left(\frac{1}{\varepsilon}\right)},
$$
and $N_{\varepsilon}^X(\mathcal{M})$ denotes the minimal number of $\varepsilon$-balls in $X$ with centres in $X$ needed to cover $\mathcal{M}$.\\
(iii) $\mathcal{M}$ exponentially attracts all bounded sets, i.e. there exists a constant $\omega>0$ such that for every bounded subset $D \subset X$ and every $t \in \mathbb{R}^+$
$$
\lim _{t \rightarrow \infty} e^{-\omega t} \operatorname{dist}_X(S(t) D, \mathcal{M})=0,
$$
where $\mathrm{dist}_X(A,B)$ denotes the Hausdorff semi-distance  between $A$ and $B$, defined as
$$
\mathrm{dist}_X(A, B)=\sup _{a \in A} \inf _{b \in B} d(a, b), \quad \text { for } A, B \subseteq X.
$$
\end{defn}

For a finite dimensional subspace $F$ of a Banach space $X$, denote by $B^F_r(x)$ and $B_r(x)$ the ball in $F$ and $X$ of center $x$ and radius $r$ respectively, that is $B^F_r(x)=\{y \in F |\|y-x\|\leq r\}$ and $B_r(x)=\{y \in X|\|y-x\|\leq r\}$. For later use, we introduce the following  covering lemma of balls in finite dimensional Banach spaces, which was proved in \cite{30}.
\begin{lem}\label{lem3.1}
For every finite dimensional subspace $F$ of a Banach space $X$, we have
 \begin{equation}\label{2.2}
N\left(r_1, B_{r_2}^F\right)\leq m 2^m\left(1+\frac{r_1}{r_2}\right)^m,
\end{equation}
for all $r_1>0, r_2>0$, where $m=\operatorname{dim} F$ and $N\left(r_1, B_{r_2}^F\right)$ is the minimum number of balls needed to cover $B_{r_2}^F$ by the ball of radius $r_1$ calculated in the metric space $X$.
\end{lem}

In order to construct the exponential attractor we need to impose  the following assumptions on the semigroup $S(t)$.

$\left(\mathcal{H}_1\right)$ For any $t\in \mathbb{R}^+$, the mapping $S(t): X\rightarrow X$ is continuous. Moreover, for the semigroup $S(t)$ there exists  a  bounded set $\mathcal{B}\subset X$ that absorbs all bounded subsets of $X$; i.e., for all bounded subsets $D \subset X$, there exists $T_{D}>0$ such that
$$
S(t) D \subset \mathcal{B}\quad \text { for all } t \geq T_{D} .
$$

$\left(\mathcal{H}_2\right)$ The bounded set $\mathcal{B} \subset X$ is positive invariant, that is, $S(t)\mathcal{B} \subseteq \mathcal{B}$ for all $t \in \mathbb{R}^+$.

$\left(\mathcal{H}_3\right)$ There is a finite dimension projection $P:X\rightarrow PX$ with a finite dimension \begin{equation}\label{2.3}
\Lambda=\dim\{PX\}
\end{equation}
and there are three positive constants $M_1, M_2, M_3$ and two constants $\lambda_0$ and $\lambda_1$ such that
 \begin{equation}\label{2.4}
\left\|PS(t)\varphi-PS(t)\psi\right\| \leq M_1e^{\lambda_0 t}\left\|\varphi-\psi\right\|
\end{equation}
and
 \begin{equation}\label{2.5}
\begin{gathered}
\left\|(I-P)S(t)\varphi-(I-P)S(t)\psi\right\|\leq (M_2e^{\lambda_1  t}+M_3e^{\lambda_0  t})\left\|\varphi-\psi\right\|
\end{gathered}
\end{equation}
for any $t\geq 0$, and $\varphi, \psi \in \mathcal{B}$.

We first construct exponential  attractors for the discrete semigroup $\{S(n) \}$.
\begin{thm}Let $\{S(n) \}$ be a discrete semigroup in $X$ and the assumptions $\left(\mathcal{H}_1\right)-\left(\mathcal{H}_3\right)$  are satisfied for discrete times $t \in \mathbb{Z}$. Moreover, assume that there exists $\alpha>0$ such that $\zeta:=\alpha e^{\lambda_0}+M_2e^{\lambda_1}+M_3e^{\lambda_0}<1$. Then, there exists an exponential attractor $\mathcal{M}$ for the semigroup $\{S(n) \}$, and the fractal dimension  is bounded  by
 \begin{equation}\label{2.6}
\begin{aligned}
\operatorname{dim}_f \mathcal{A}\leq \frac{\Lambda[\ln\Lambda +\ln(2+\frac{M_1}{\alpha })]}{-\ln \zeta}<\infty.
\end{aligned}
\end{equation}
\end{thm}
\begin{proof}
\textbf{1) Covering of} $S(m)\mathcal{B}$. We   construct the covering of $U(m)\mathcal{B}$ by inductively defining a family of sets  $W^m, m\in \mathbb{N}$ that  satisfy the following properties
 \begin{equation}\label{2.7}
\left\{\begin{array}{l}
(W1) \quad W^m \subset S(m) \mathcal{B}\subset \mathcal{B},\\
(W2) \quad \sharp W^m \leq [\Lambda 2^\Lambda \left(1+\frac{M_1}{\alpha }\right)^\Lambda]^m,\\
(W3) \quad S(m) \mathcal{B}\subset \bigcup_{u \in W^m } B_{\zeta^m R_{\mathcal{B}}}(u),
\end{array}\right.
\end{equation}
where $\sharp W^m$ represents the number of elements of $W^m$.

We first consider the case $m=1$, i.e., construct a covering of the image of  $S(1)\mathcal{B}$. Since $\mathcal{B}$ is bounded by $\left(\mathcal{H}_1\right)$,  there exists a constant $R_{\mathcal{B}}$ such that  $R_{\mathcal{B}}:=\sup _{u \in \mathcal{B}}\|u\|_X$. Then for any $u_{1}\in \mathcal{B}$, we have
$\mathcal{B}\subset B_{R_{\mathcal{B}}}\left(u_{1}\right)$. For any $u \in \mathcal{B}\cap B\left(u_{1}, R_{\mathcal{B}}\right)$, it follows from $\left(\mathcal{H}_3\right)$ that
\begin{equation}\label{2.8}
\begin{gathered}
\left\|PS(1) u-P S(1) u_1\right\| \leq M_1e^{\lambda_0}R_{\mathcal{B}},
\end{gathered}
\end{equation}
and
 \begin{equation}\label{2.9}
\begin{gathered}
\left\|(I-P) S(1) u-(I-P) S(1) u_1\right\| \leq  (M_2e^{\lambda_1}+M_3e^{\lambda_0 })R_{\mathcal{B}}.
\end{gathered}
\end{equation}
By Lemma \ref{lem3.1}, we can find $y_{1}^1, \ldots, y_{1}^{n_1}\in PS(1)\mathcal{B}$ and $\alpha>0$ such that
 \begin{equation}\label{2.10}
\begin{gathered}
B_{P X}\left(P S(1) u_1, M_1e^{\lambda_0}R_{\mathcal{B}}\right) \subset \bigcup_{j=1}^{n_1} B_{PX}\left(y_{1}^j, \alpha e^{\lambda_0}R_{\mathcal{B}}\right)
\end{gathered}
\end{equation}
with
 \begin{equation}\label{2.11}
\begin{gathered}
n_1\leq \Lambda 2^\Lambda \left(1+\frac{M_1}{\alpha }\right)^\Lambda,
\end{gathered}
\end{equation}
where $\Lambda$ is the dimension of $P X$ and we have denoted by $B_{PX}(y, r)$ the ball in $P X$ of radius $r$ and center $y$.
Set
 \begin{equation}\label{2.12}
\begin{gathered}
u_{1}^j=y_{1}^j+(I-P) S(1) u_{1}
\end{gathered}
\end{equation}
for $ j=1, \ldots, n_1$ and $W^1=\{u_{1}^1, u_{1}^2,\cdots, u_{1}^{n_1}\}$. Then, for any $u \in \mathcal{B}\cap B\left(u_{1}, R_{\mathcal{B}}\right)$, there exists a $j$ such that
 \begin{equation}\label{2.13}
\begin{aligned}
\left\|S(1) u-u_{1}^j\right\|
& \leq\left\|P S(1) u-y_{1}^j\right\|+\left\|(I-P) S(1) u-(I-P) S(1) u_{1}\right\| \\
& \leq\left(\alpha e^{\lambda_0}+M_2e^{\lambda_1}+M_3e^{\lambda_0}\right)R_{\mathcal{B}}\\
& \leq\zeta_1R_{\mathcal{B}},
\end{aligned}
\end{equation}
indicating $(W3)$ is satisfied for $m=1$. Furthermore, it is clear from the definition of $W^1$ that it satisfies $(W1)$ and $(W2)$.
This completes the proof of case $m=1$.

Assume that the sets $W^l$ satisfying \eqref{2.7} have been already constructed for all $m\leq l$, i.e.,
there exists covering
 \begin{equation}\label{2.14}
S(l) \mathcal{B}\subset \bigcup_{u \in W^l} B_{\zeta^{l}R_{\mathcal{B}}}(u).
\end{equation}
We construct  in the sequel the covering of $W^{l+1}$ satisfies \eqref{2.7}.  By the semigroup property, we have
 \begin{equation}\label{2.15}
\begin{aligned}
S(l+1) \mathcal{B}& =S(1) S(l) \mathcal{B} \subset \bigcup_{u \in W^l} S(1) B_{\zeta^lR_{\mathcal{B}}}(u),
\end{aligned}
\end{equation}
that is $S(l+1) \mathcal{B}$ can be covered by $\bigcup_{u \in W^l} S(1) B_{\zeta^lR_{\mathcal{B}}}(u)$. We construct in the following a covering of $\bigcup_{u \in W^l} S(1) B_{\zeta^lR_{\mathcal{B}}}(u)$.

Let $u_l \in W^l$. It follows from the induction hypothesizes  $(W1)$ and $(W3)$ that
 \begin{equation}\label{2.16}
u_l \in S(l) \mathcal{B}\subset \bigcup_{u \in W^l} B_{\zeta^{l}R_{\mathcal{B}}}(u).
\end{equation}
 Therefore, for any $u \in \mathcal{B}\cap \bigcup_{u \in W^l} S(1) B_{\zeta^lR_{\mathcal{B}}}(u)$, it follows from $\left(\mathcal{H}_3\right)$ that
\begin{equation}\label{2.17}
\begin{gathered}
\left\|PS(1) u-P S(1) u_l\right\| \leq M_1e^{\lambda_0}\zeta^lR_{\mathcal{B}},
\end{gathered}
\end{equation}
and
 \begin{equation}\label{2.18}
\begin{gathered}
\left\|(I-P) S(1) u-(I-P) S(1) u_l\right\| \leq  (M_2e^{\lambda_1}+M_3e^{\lambda_0 })\zeta^lR_{\mathcal{B}}.
\end{gathered}
\end{equation}
By Lemma \ref{lem3.1}, we can find $y_l^1, \ldots, y_l^{n_l}\in S(1) S(l) \mathcal{B}$ such that
 \begin{equation}\label{2.19}
\begin{gathered}
B_{P X}\left(P S(1) u_l, \alpha e^{\lambda_0 }R_{\mathcal{B}}\right) \subset \bigcup_{j=1}^{n_l} B_{PX}\left(y_l^j, M_1e^{\lambda_0}\zeta^lR_{\mathcal{B}}\right)
\end{gathered}
\end{equation}
with
 \begin{equation}\label{2.20}
\begin{gathered}
n_l\leq \Lambda 2^\Lambda \left(1+\frac{M_1}{\alpha}\right)^\Lambda,
\end{gathered}
\end{equation}
where $\Lambda$ is the dimension of $P X$.
Set
 \begin{equation}\label{2.21}
\begin{gathered}
u_l^j=y_l^j+(I-P) S(1) u_l
\end{gathered}
\end{equation}
for $ j=1, \ldots, n_l$. Then, there exists a $j$ such that
 \begin{equation}\label{2.22}
\begin{aligned}
\left\|S(1) u-u_l^j\right\|
& \leq\left\|P S(1) u-y_l^j\right\|+\left\|(I-P) S(1) u-(I-P) S(1) u_l\right\| \\
& \leq\left(\alpha e^{\lambda_0}+M_2e^{\lambda_1}+M_3e^{\lambda_0}\right)\zeta^lR_{\mathcal{B}}=\zeta^{l+1}R_{\mathcal{B}}.
\end{aligned}
\end{equation}
This implying that $ \bigcup_{u \in W^l} S(1) B_{\zeta^lR_{\mathcal{B}}}(u)$ is  covered by balls with radius $\zeta^{l+1}R_{\mathcal{B}}$ and centers $\{u_l^1, u_l^2, \cdots, u_l^{n_l}\}$ and hence $(W3)$ holds. Denote the new set of centres by $W^{l+1}$. From the  induction hypothesis, we have $\sharp W^{l}\leq [\Lambda 2^\Lambda \left(1+\frac{M_1}{\alpha }\right)^\Lambda]^l$, which yields $\sharp W^{l+1} \leq n_l \sharp W^{l} \leq [\Lambda 2^\Lambda \left(1+\frac{M_1}{\alpha }\right)^\Lambda]^{l+1}$ and proves $(W2)$. By construction the set of centres $W^{l+1}$, we can see $W^{l+1} \subset S(1) S(l) \mathcal{B}=S(l+1)\mathcal{B}$,  which concludes the proof of the properties $(W1)$.

\textbf{2) Construction of random exponential attractor for $\{S(n) \}$.}
We define $E^1:=W^1$  and set
 \begin{equation}\label{2.24}
\begin{aligned}
E^{n+1}:=W^{n+1}\cup S(1) E^{n}, \quad n \in \mathbb{Z}^+.
\end{aligned}
\end{equation}
Then, if follows from the definition of the sets $E^n(k)$, the properties of the sets $W^n(k)$ and the positive invariance of the absorbing set $\mathcal{B}$ that the family of sets $E^n, n \in \mathbb{Z}^+$ satisfies\\
(E1) $\quad S(1) E^{n-1} \subset E^{n}, \quad E^n \subset S(n)\mathcal{B}$,\\
(E2) $\quad E^n=\bigcup_{i=0}^n S(l) W^{n-i}, \quad \sharp E^n \leq \sum_{i=0}^n (\Lambda 2^\Lambda \left(1+\frac{M_1}{\alpha }\right)^\Lambda)^i$,\\
(E3) $\quad S(n) \mathcal{B}\subset \bigcup_{u \in E^n} B_{\zeta^n R_{\mathcal{B}}}(u)$.

Based on the family of sets $E^n$, we define $\mathcal{M}:=\overline{\bigcup_{n \in \mathbb{Z}^+} E^n}$ and show that it yields an exponential  attractor for the semigroup $\{S(n)\}$.

\textbf{Positive invariance of the $\mathcal{M}$.} It follows from property $(E 1)$ that for all $l \in  \mathbb{Z}^+$, we have
  \begin{equation}\label{2.25}
\begin{aligned}
S(l) \bigcup_{n \in \mathbb{Z}^+} E^n &=\bigcup_{n \in \mathbb{N}^+} S(l) E^n\subset \bigcup_{n \in \mathbb{N}^+} E^{n+l}\subset \bigcup_{n \in \mathbb{N}^+} E^n.
\end{aligned}
\end{equation}
Thanks to the continuous property in $\left(\mathcal{H}_1\right)$, we can take closure in both sides of \eqref{2.25}, giving rise to
 \begin{equation}\label{2.26}
\begin{aligned}
S(l) \mathcal{M}& :=S(l) \overline{\bigcup_{n \in \mathbb{Z}^+} E^n} &=\overline{\bigcup_{n \in \mathbb{N}^+} S(l) E^n}\subset \overline{\bigcup_{n \in \mathbb{N}^+} E^{n+l}}\subset \overline{\bigcup_{n \in \mathbb{N}^+} E^n}=\mathcal{M}.
\end{aligned}
\end{equation}

\textbf{Compactness and finite dimensionality of $\mathcal{M}$.} We  prove in the sequel that  the set $\mathcal{M}$ is non-empty, precompact and of finite fractal dimension.  It follows from $(E1)$ that  for any $l \in \mathbb{Z}^+$ and $m\geq l+1$, it holds
 \begin{equation}\label{2.27}
\begin{aligned}
E^l & \subset S(l) B.
\end{aligned}
\end{equation}
Thus, for any $l \in \mathbb{Z}^+$ we have
 \begin{equation}\label{2.28}
\begin{aligned}
\bigcup_{n \in \mathbb{N}^+} E^n=\bigcup_{n=0}^l E^n \cup \bigcup_{n=l+1}^{\infty} E^n \subset \bigcup_{n=0}^l E^n \cup S(l) \mathcal{B}.
\end{aligned}
\end{equation}
Due to $\zeta<1$, for any give $\varepsilon>0$, there exists $l \in \mathbb{N}$ such that
 \begin{equation}\label{2.29}
\begin{aligned}
\zeta^{l+1} R_{\mathcal{B}} \leq \varepsilon<\zeta^{l}  R_{\mathcal{B}},
\end{aligned}
\end{equation}
which combined with the fact
 \begin{equation}\label{2.30}
\begin{aligned}
S(l) \mathcal{B}\subset \bigcup_{u \in W^l} B_{\varepsilon}(u) .
\end{aligned}
\end{equation}
indicates that the estimate of the number of $\varepsilon$-balls in $X$ needed to cover $bigcup_{n\in \mathbb{Z}^+} E^n$ is
 \begin{equation}\label{2.31}
\begin{aligned}
N_{\varepsilon}\left(\bigcup_{n\in \mathbb{Z}^+} E^n\right)& \leq \sharp\left(\bigcup_{n=0}^l E^l\right)+\sharp W^l \leq(l+1) \sharp E^l+(\Lambda 2^\Lambda \left(1+\frac{M_1}{\alpha }\right)^\Lambda)^l \\
& \leq(l+1)^2 [\Lambda 2^\Lambda \left(1+\frac{M_1}{\alpha }\right)^\Lambda]^l.
\end{aligned}
\end{equation}
 This proves the precompactness of $\bigcup_{n\in \mathbb{Z}^+} E^n$ in $X$, which directly implies the closure $\mathcal{M}:=\overline{\bigcup_{n\in \mathbb{Z}^+} E^n}$ is compact in $X$ since  $X$ is a Banach space.

It follows from \eqref{2.29} and \eqref{2.31} that   the fractal dimension of the set  $\mathcal{M}$ can be estimated by
 \begin{equation}\label{2.32}
\begin{aligned}
\operatorname{dim}_f \mathcal{M} & =\limsup_{\varepsilon \rightarrow 0} \frac{\ln N_{\varepsilon}(\mathcal{M})}{-\ln \varepsilon}\\
& \leq \limsup_{l \rightarrow \infty} \frac{\ln (l+1)^2+ \ln [\Lambda 2^\Lambda \left(1+\frac{M_1}{\alpha }\right)^\Lambda]^l}{-\ln (\zeta^l R_{\mathcal{B}})}\\
&=\frac{\Lambda[\ln\Lambda +\ln(2+\frac{M_1}{\alpha })]}{-\ln \zeta}<\infty.
\end{aligned}
\end{equation}

\textbf{3) Exponential attraction of $\mathcal{M}$.} We are left to show that the set $\mathcal{M}$ exponentially attracts all bounded subsets of $X$ at time $l \in \mathbb{Z}^+$. It follows from assumptions $\left(\mathcal{H}_1\right)$ that for any bounded subset $D \subset X$, there exists an $n_{D} \in \mathbb{Z}^+$ such that $S(l) D \subset \mathcal{B}$ for all $l \geq n_{D}$.
If $l \geq n_{D}+1$, that is $l\geq n_{D}+n_0$ with some $n_0 \in \mathbb{N}$, then
 \begin{equation}\label{2.33}
\begin{aligned}
\operatorname{dist}_{X} \left(S(l) D, \mathcal{M}\right)&=\operatorname{dist}_{X}\left(S(l) D, \overline{\bigcup_{n=0}^{\infty}E^n}\right)  \\& \leq \operatorname{dist}_{X}\left(S\left(n_0\right) S\left(l-n_0\right) D, \bigcup_{n=0}^{\infty} E^n\right) \\
& \leq \operatorname{dist}_{X}\left(S\left(n_0\right) \mathcal{B}, \bigcup_{n=0}^{\infty} E^n\right) \\
& \leq \operatorname{dist}_{X}\left(S\left(n_0\right) \mathcal{B}, E^{n_0}\right) \\
& \leq \zeta^{n_0} R_{ \mathcal{B}} \leq c e^{-\omega n}
\end{aligned}
\end{equation}
for some constants $c \geq 0$ and $\omega>0$. This completes the proof.
\end{proof}

By adopting the same procedure as the proof of Theorem 3.2 in \cite{C9}, we have  the following  results about the existence of continuous semigroup in Banach spaces.

\begin{thm}\label{thm2.1} For the above semigroup $S(t)$ on the Banach space $X$. We assume that assumptions  $\left(H_1\right)$-$\left(H_3\right)$hold.Then, there exists an exponential attractor $\mathcal{M}$ for the semigroup $\{S(t) \}$, and the fractal dimension  is bounded  by
 \begin{equation}\label{2.34}
\begin{aligned}
\operatorname{dim}_f \mathcal{A}\leq \frac{\Lambda[\ln\Lambda +\ln(2+\frac{M_1}{\alpha })]}{-\ln \zeta}<\infty.
\end{aligned}
\end{equation}
\end{thm}
\section{Applications}
In this section, we are concerned about applications of the above established theoretical results to a the retarded functional differential equation (RFDE)  and a retarded reaction-diffusion equation.
\subsection{Retarded functional differential  equation}
Consider the following nonlinear autonomous RFDE
 \begin{equation}\label{3.1}
\begin{aligned}
\dot{x}(t)=Ax(t)+bx(t-\tau)+f(x_t),\\
x_0=\phi,
\end{aligned}
\end{equation}
where $x(t)\in \mathbb{R}^n$, $A\in \mathbb{R}^{n\times n}$, $b$ is a constant, $r$ stands for the delay and $f$ is a  continuous nonlinear mapping from $X$ into $\mathbb{R}^n$.  The initial condition $\phi\in \mathcal{C}\triangleq C([-r,0],\mathbb{R}^n)$, with $\mathcal{C}$ being the Banach space of continuous functions from $[-r,0]$ to $\mathbb{R}^n$ equipped with the supremum norm  $\|\phi\|_{X}=\sup_{\theta \in[-r, 0]}|\phi(\theta)|$ for any $\phi \in X$ and $|\cdot|$ is the usual norm of $\mathbb{R}^n$.  For notation simplicity, define the linear part of \eqref{3.1} as a  linear  mapping $L$ from $X$ into $\mathbb{R}^n$ given by
 \begin{equation}\label{3.1a}
\begin{aligned}
L\phi =A\phi(0)+b\phi(\tau)
\end{aligned}
\end{equation}
for any $\phi\in X$.

Since $L$ is linear,  there exists an $n \times n$ matrix $\eta(\theta),-r \leq \theta \leq 0$, whose elements are of bounded variation, normalized so that $\eta$ is continuous from the left on $(-r, 0)$ and $\eta(0)=0$, such that,
 \begin{equation}\label{3.2}
\begin{aligned}
L \phi=\int_{-r}^0 d[\eta(\theta)] \phi(\theta), \quad \phi \in X.
\end{aligned}
\end{equation}
Denote by $u^\phi_t$ the solution of the liner part
 \begin{equation}\label{3.3}
\begin{aligned}
\dot{u}(t)=L x_t,\\
u_0=\phi.
\end{aligned}
\end{equation}
Then the solution operator defined by $u^\phi_t=S(t)\phi$  is a strongly continuous semigroup with infinitesimal generator
 \begin{equation}\label{3.4}
\begin{aligned}
A \phi & =\frac{d \phi}{d \theta},\\
\mathcal{D}(A) & =\left\{\phi \in C: \frac{d \phi}{d \theta} \in C, \frac{d \phi}{d \theta}(0)=\int_{-r}^0 d[\eta(\theta)] \phi(\theta)\right\}.
\end{aligned}
\end{equation}
Furthermore, $S(t)$ is compact for $t \geq r$. It follows from Lemma 2.1 in Chapter 7 in \cite{JH} that the operator  $A$ defined by Equation \eqref{3.4} has only point spectrum, i.e., $\sigma(A)=\mathrm{P} \sigma(A)$ and $\lambda$ is in $\sigma(A)$ if and only if $\lambda$ satisfies the characteristic equation
 \begin{equation}\label{3.5}
\begin{aligned}
\operatorname{det} \Delta(\lambda)=0, \quad \Delta(\lambda)=\lambda I-\int_{-r}^0 e^{\lambda \theta} d \eta(\theta).
\end{aligned}
\end{equation}

Denote by $\Sigma=\left\{\varrho_1>\varrho_2>\ldots>\varrho_m\right\}$ as a finite set of characteristic value of equation \eqref{3.5} with  multiplicity $n_1, n_2,\cdots$, where $\varrho_1$ is defined as
 \begin{equation}\label{3.6}
\begin{aligned}
\varrho_1=\max \left\{\operatorname{Re} \lambda: \Delta(\lambda)\xi=\left[\Delta(\lambda)=\lambda I-\int_{-r}^0 e^{\lambda \theta} d \eta(\theta)\right] \xi=0\right\}.
\end{aligned}
\end{equation}
It follows from Theorem 6.1 and Lemma 2.1 in Chapter 7 in \cite{JH} that for any given $\varrho_m<0$, $m\geq 1$, there is a
 \begin{equation}\label{3.7}
\begin{aligned}
k_m=n_1+n_2+\cdots+n_m
\end{aligned}
\end{equation}
dimensional  subspace $X^U_{k_m}$ such that
$$X=X^U_{k_m} \bigoplus X^S_{k_m}$$
is the decomposition of $X$ by $\varrho_m$. Let $P_{k_m}$ and $Q_{k_m}$ be the projection of $X$ onto $X^U_{k_m}$ and $ X^S_{k_m}$ respectively, that is $X^U_{k_m}=P_{k_m}X$, $X^S_{k_m}=(I-P_{k_m})X=Q_{k_m}X$. It follows from the definition $P_{k_m}$ and $Q_{k_m}$ that there exists a positive constant $K$ such that
\begin{equation}\label{3.8}
\begin{aligned}
\left\|S(t)Q_{k_m} x\right\| & \leq K e^{\varrho_m t}\|x\|, & & t \geq0.
\end{aligned}
\end{equation}

For the purpose of showing the squeeze property, we extend the domain of $U(t)$ to the following space of some discontinuous functions
 \begin{equation}\label{3.9}
\begin{aligned}
\hat{C}=\left\{\phi:[-r, 0] \rightarrow X ;\left.\phi\right\|_{[-r, 0)} \quad \text {is continuous and } \lim _{\theta \rightarrow 0^{-}} \phi(\theta) \in X \quad \text{exists} \right\}
\end{aligned}
\end{equation}
and introduce the following informal variation of  constant formula established in \cite{CM}
 \begin{equation}\label{3.10}
\begin{aligned}
u^\phi_t & =S(t) \phi+\int_0^t S(t-s) X_0 f(u_s) d s, \quad t \geq 0,
\end{aligned}
\end{equation}
where $X_0:[-r, 0] \rightarrow B(X, X)$ is given by $X_0(\theta)=0$ if $-r \leq \theta<0$ and $X_0(0)=I d$.

\begin{rem}
In general, the solution semigroup defined by \eqref{3.10} have no definition at discontinuous functions and the integral in the formula is undefined as an integral in the phase space. However,  if interpreted correctly, one can see that \eqref{3.10} does make sense. Details can be found in \cite{CM} Pages 144 and 145.
\end{rem}

In the reminder of this paper, we always assume that $\varrho_1<0$ and hence we have the following results by Theorem 6.2 on page 24 in \cite{JH}.
 \begin{lem}\label{lem4}If $\varrho_1<0$, then there exist positive constants $\gamma<-\varrho_1$ and $K_{0}(\gamma)$ such that
\begin{equation}\label{3.11}\|S(t)\phi\|<K_{0} \mathrm{e}^{-\gamma t}\|\phi\|\end{equation}
 for all $t \geq 0$.
\end{lem}
Define the nonlinear dynamical system generated by \eqref{3.1} by $\Phi(t)\phi=u^\phi_t$ for any $\phi\in X$. It follows from Theorem 2.2 on page 24 in \cite{JH}, $\Phi(t): X\rightarrow X$ is continuous for any $t\in \mathbb{R}^+$. In the following, we construct exponential attractors for of $\Phi$. We first show that $\Phi$ admits a positive invariant absorbing set,  i.e. $\left(\mathcal{H}_1\right)$ and $\left(\mathcal{H}_2\right)$ are satisfied. We make the following assumption on the nonlinear term $f$.

$\mathbf{Hypothesis \  A1}$ $\left\|f\left(\phi_1\right)-f\left(\phi_2\right)\right\| \leq L_f\left\|\phi_1-\phi_2\right\| \text { for any } \phi_1, \phi_2 \in X.$

\begin{thm}\label{thm3.1}
Assume that $\mathbf{Hypothesis \  A1}$ holds and $\varrho_1<0$. Choose $0<\gamma<-\varrho_1$ such that $K_0<1$ and assume that $K_0L_f-\gamma<0$, then the dynamical system $\Phi$ admits a invariant absorbing set $\mathcal{B}$ defined by
\begin{equation}\label{3.12}
\mathcal{B}=\{\phi \in C| \|\phi\|\leq \frac{1}{1-K_0}[\frac{K_0L_ff(0)}{\gamma}+\frac{1}{\gamma-K_0L_f}]\}.
\end{equation}
\end{thm}
\begin{proof}
It follows from \eqref{3.10} that
\begin{equation}\label{3.13}
\begin{aligned}
\left\|u_t \right\|\leq &\left\|\left(S(t)\phi\right)\right\|+\|\int_{0}^{t}\left[S(t-s) X_0 f(u_s)\right]\mathrm{d}s\| \\
\leq & K_0e^{-\gamma t}\left\|\phi\right\|+K_0L_f\int_{0}^{t} e^{-\gamma(t-s) }(\|u_s\|+f(0)) \mathrm{d} s\\
\leq & K_0e^{-\gamma t}\left\|\phi\right\|+K_0L_f\int_{0}^{t} e^{-\gamma(t-s) }\|u_s\| \mathrm{d} s+\frac{K_0L_ff(0)(1-e^{-\gamma t})}{\gamma}.
\end{aligned}
\end{equation}
Multiply both sides of \eqref{3.13} by $e^{\gamma t}$ gives
\begin{equation}\label{3.14}
\begin{aligned}
 e^{\gamma t}\left\|u_t \right\|\leq & K_0 \left\|\phi\right\|+K_0L_f\int_{0}^{t} e^{\gamma s}\|u_s\| \mathrm{d} s+\frac{K_0L_ff(0)e^{\gamma t}}{\gamma}.
\end{aligned}
\end{equation}
Applying Gr\"{o}nwall's inequality yields
\begin{equation}\label{3.15}
\begin{aligned}
 e^{\gamma t}\left\|u_t \right\| \leq & K_0\left\|\phi\right\|e^{K_0L_f t}+\frac{K_0L_ff(0)e^{\gamma t}}{\gamma}+\frac{e^{(\gamma-K_0L_f) t}}{\gamma-K_0L_f},
\end{aligned}
\end{equation}
and hence
\begin{equation}\label{3.16}
\begin{aligned}
\left\|u_t \right\| \leq & K_0\left\|\phi\right\|e^{(K_0L_f-\gamma) t}+\frac{K_0L_ff(0)}{\gamma}+\frac{e^{-K_0L_f t}}{\gamma-K_0L_f}\\ \leq & K_0\left\|\phi\right\|e^{(K_0L_f-\gamma) t}+\frac{K_0L_ff(0)}{\gamma}+\frac{1}{\gamma-K_0L_f}.
\end{aligned}
\end{equation}
Therefore, in the case $K_0L_f-\gamma<0$, for any $\phi\in X$, there exists a $t_{\|\phi\|}>0$ such that
\begin{equation}\label{3.17}
\begin{aligned}
\left\|u_t \right\| \leq & \frac{1}{1-K_0}[\frac{K_0L_ff(0)}{\gamma}+\frac{1}{\gamma-K_0L_f}].
\end{aligned}
\end{equation}
That is, $\mathcal{B}$ is an absorbing set for $\Phi$. Indeed, for any bounded subset $D\subset X$, denote by $r_D=\sup _{u \in D}\|u\|_X$, if we take $T_{D}=\frac{1}{\gamma}\ln \frac{r_D\gamma (1-K_0)(\gamma-K_0L_f)}{K_0L_ff(0)(\gamma-K_0L_f)+\gamma}$, then we have
$$
\Phi(t) D \subset \mathcal{B}
$$
for all $t \geq T_{D}$.

The invariance property clearly follows since for any $\phi\in \mathcal{B}$, by \eqref{3.16} and \eqref{3.17}, we have
 \begin{equation}\label{3.18}
\begin{aligned}
\left\|\Phi(t)\phi\right\| =&\left\|u_t \right\|\leq K_0\left\|\phi\right\|e^{(K_0L_f-\gamma) t}+\frac{K_0L_ff(0)}{\gamma}+\frac{e^{-K_0L_f t}}{\gamma-K_0L_f}
 \\ \leq & (\frac{K_0}{1-K_0}+1)[\frac{K_0L_ff(0)}{\gamma}+\frac{1}{\gamma-K_0L_f}] \\ \leq & \frac{1}{1-K_0}[\frac{K_0L_ff(0)}{\gamma}+\frac{1}{\gamma-K_0L_f}].
\end{aligned}
\end{equation}
This completes the proof.
\end{proof}

Subsequently, we prove the squeeze property of $\Phi$, i.e., $\left(\mathcal{H}_3\right)$ holds.
\begin{thm}\label{thm3.2}Let $P$ be  the finite dimension projection $P_{k_m}$ be defined by \eqref{3.8}, $K, \varrho_{m},\gamma$ and $K_0$ being defined in \eqref{3.6}, \eqref{3.7} and \eqref{3.11} respectively, then we have
 \begin{equation}\label{3.19}
\left\|P \Phi(t)\varphi-P \Phi(t)\psi\right\| \leq 2e^{(L_fK_0-\gamma)t}\left\|\varphi-\psi\right\|
\end{equation}
and
 \begin{equation}\label{3.20}
\begin{gathered}
\left\|(I-P) \Phi(t)\varphi-(I-P) \Phi(t)\psi\right\|\leq (Ke^{\varrho_m t}+\frac{KL_fK_0}{-\gamma+L_f-\varrho_m} e^{(L_f-\gamma)t})\left\|\varphi-\psi\right\|
\end{gathered}
\end{equation}
for any $t\geq 0$ and $\varphi, \psi\in\mathcal{B}$.
\end{thm}
\begin{proof}
For any $\varphi, \psi\in X$, denote by $y=\varphi-\psi$ and $w_t=\Phi(t)\varphi-\Phi(t)\psi=u^\varphi_t-u^\psi_t$. Then it follows from \eqref{3.10} that
 \begin{equation}\label{3.21}
\begin{aligned}
w_t& =S(t) y+\int_0^t S(t-s) X_0 [f(u^\varphi_s)-f(u^\psi_s)]d s, \quad t \geq 0.
\end{aligned}
\end{equation}
Take projection on $I-P$ on both sides of \eqref{3.21} leads to
\begin{equation}\label{3.22}
\begin{aligned}
\|(I-P)w_t\|_X=&\|(I-P)S(t)y+\int_0^t (I-P)S(t-s) X_0 [f(u^\varphi_s)-f(u^\psi_s)]d s\|_X\\
\leq & K e^{\varrho_m t}\|y\|_X +L_fK_0 \int_0^t  e^{-\gamma(t-s)}\|(I-P)w_s\|_Xd s.
\end{aligned}
\end{equation}
Multiply both sides of \eqref{3.22} by $e^{\gamma t}$ gives
\begin{equation}\label{3,23}
\begin{aligned}
e^{\gamma t}\|(I-P)w_t\|_X \leq & Ke^{(\varrho_m+\gamma) t}\|y\|_X +L_fK_0 \int_0^t  e^{\gamma s}\|(I-P)w_t\|_Xd s.
\end{aligned}
\end{equation}
By applying the Gronwall inequality, we have
\begin{equation}\label{3.24}
\begin{aligned}
e^{\gamma t}\|(I-P)w_t\|_X \leq & \|y\|[Ke^{(\varrho_m+\gamma) t} +\frac{KL_fK_0 }{\varrho_m+\gamma-L_fK_0}(e^{(\varrho_m+\gamma) t}-e^{L_fK_0})],
\end{aligned}
\end{equation}
indicating that
\begin{equation}\label{3.25}
\begin{aligned}
\|(I-P)w_t\|_X \leq & \|y\|[Ke^{\varrho_m t} +\frac{KL_fK_0}{\varrho_m+\gamma-L_fK_0}(e^{\varrho_m  t}-e^{(L_fK_0-\gamma)t})]\\
 \leq & \|y\|[Ke^{\varrho_m t} +\frac{KL_fK_0 }{-\gamma+L_fK_0-\varrho_m}e^{(L_fK_0-\gamma)t}].
\end{aligned}
\end{equation}
Hence, the second part holds with $\lambda_0=L_fK_0-\gamma$, $\lambda_1=\varrho_m$, $M_2=K$ and $M_3=\frac{KL_fK_0 }{-\varrho_m-\gamma+L_fK_0}$.

Subsequently, we prove the first part. Since $S(t)y=PS(t)y+(I-P)S(t)y$, we have
\begin{equation}\label{3.26}
\begin{aligned}
\|PS(t)y\|_X\leq &\|S(t)y\|+\|(I-P)S(t)y\|_X.
\end{aligned}
\end{equation}
Take projection of $P$ on both sides of \eqref{3.10} and take into account of \eqref{3.26} gives
\begin{equation}\label{3.27}
\begin{aligned}
\|Pw_t\|_X=&\|S(t)y\|+\|(I-P)S(t)y\|+\int_0^t \|PS(t-s) X_0 [f(u^\varphi_s)-f(u^\psi_s)]d s\|_X\\
\leq &( e^{-\gamma t}+e^{\varrho_mt})\|y\|+L_fK_0 \int_0^t  e^{-\gamma(t-s)}\|Pw_t\|_Xd s\\
\leq &2e^{-\gamma t}\|y\|+L_fK_0 \int_0^t  e^{-\gamma (t-s)}\|Pw_t\|_Xd s.
\end{aligned}
\end{equation}
Multiply both sides of \eqref{3.28} by $e^{\gamma t}$ gives
\begin{equation}\label{3.28}
\begin{aligned}
e^{\gamma t}\|Pw_t\|_X \leq & 2\|y\| +L_fK_0\int_0^t  e^{\gamma s}\|Pw_t\|_Xd s.
\end{aligned}
\end{equation}
By applying the Gronwall inequality, we have
\begin{equation}\label{3.29}
\begin{aligned}
e^{\gamma t}\|Pw_t\|_X \leq & 2\|y\|e^{L_fK_0 t},
\end{aligned}
\end{equation}
indicating that
\begin{equation}\label{3.30}
\begin{aligned}
\|Pw_t\|_X \leq & 2\|y\|e^{(L_fK_0-\gamma) t}.
\end{aligned}
\end{equation}
Hence, the first part holds by taking $M_1=2$ and $\lambda_0=L_fK_0-\gamma$.
\end{proof}

 Hence, it follows from Theorem \ref{thm2.1} that we have the following results about existence of an exponential  attractor $\mathcal{M}$ for the nonlinear dynamical system $\Phi$ generated by \eqref{3.1}.

\begin{thm}\label{thm3.3}
Let $k_m, \varrho_{1}, \varrho_{m},\gamma$ and $K$ be  defined in \eqref{3.8},  \eqref{3.6} and \eqref{3.7} respectively,  $P$ be  the finite dimension projection $P_{k_m}$ be defined by \eqref{3.11}. Assume that the conditions of Theorem \ref{thm3.1} are satisfied. Moreover, assume   there exists $\alpha>0$ such that $\zeta:=\alpha e^{(L_fK_0-\gamma)}+Ke^{\varrho_m}+\frac{KL_fK_0 }{-\gamma+L_fK_0-\varrho_m}e^{(L_fK_0-\gamma)}<1$, then, \eqref{3.1} admits an exponential  attractor $\mathcal{M}$ of which the fractal dimension has an upper bound
 \begin{equation}\label{3.31}
\operatorname{dim}_f \mathcal{M}\leq \frac{k_m[\ln k_m +\ln(2+\frac{2}{\alpha })]}{-\ln \zeta}<\infty.
\end{equation}
\end{thm}

\begin{rem}\label{rem5.2}
Since $\varrho_1$ and $\varrho_m$ represent the first and the $m$-th eigenvalues the linear part $L$ of Eq. \eqref{3.1}, which contains the delay effect, we can see the Hausdorff dimension of exponential attractor $\mathcal{M}$ of Eq. \eqref{3.1} depend on the time delay via the distribution of eigenvalues of the linear part $L$ of Eq. \eqref{3.1}. Furthermore, it follows from \eqref{3.31} that the fractal dimension depends on the constants of exponential dichotomy, the Lipschitz constant of the nonlinear term and the spectrum gap of the liner part $L$, indicating that the fractal dimension of global attractor $\mathcal{M}$ is very flexible to be tuned by a variety of parameters.
\end{rem}

We have the following special case about the  fractal dimension of global attractor $\mathcal{M}$ in the case $m=1$.

\begin{cor}\label{cor3.1}
Let $k_m, K_0, \varrho_{m},\gamma$ and $K$ be  defined in \eqref{3.8},  \eqref{3.6} and \eqref{3.7} respectively,  $P$ be  the finite dimension projection $P_{k_m}$ be defined by \eqref{3.11}. Assume that the conditions of Theorem \ref{thm3.1} are satisfied and $k_1=1$. Moreover, assume   there exists $\alpha>0$ such that $\zeta:=\alpha e^{(L_fK_0-\gamma)}+Ke^{\varrho_1}+\frac{KL_fK_0 }{-\gamma+L_fK_0-\varrho_1}e^{(L_fK_0-\gamma)}<1$, then, \eqref{3.1} admits an exponential  attractor $\mathcal{M}$ of which the fractal dimension has an upper bound
 \begin{equation}\label{3.34g}
\operatorname{dim}_f \mathcal{M}\leq \frac{\ln(2+\frac{2}{\alpha })}{-\ln [(\alpha+K) e^{(L_fK_0+\varrho_1)}+Ke^{\varrho_1}]}<\infty.
\end{equation}
\end{cor}

\subsection{Retarded reaction-diffusion equation}
This subsection is devoted to the existence of exponential attractors  for  retarded reaction-diffusion equations. We  consider the following autonomous equation on bounded domain with Dirichlet boundary condition
 \begin{equation}\label{5.1}
 \left\{\begin{array}{l}
\frac{\partial}{\partial t} u(x, t)  =\frac{\partial^2}{\partial x^2} u(x, t)-au(x, t)-bu(x, t-r)+f(u(x, t-r)), 0 \leq x \leq \pi, t \geq 0, \\ u(0, t)  =u(\pi, t)=0, t \geq 0, \\ u(x, t) =\phi(t)(x), 0 \leq x \leq \pi,-r \leq t \leq 0.
\end{array}\right.
\end{equation}
where $a$, $b$ and $r$ are positive constants. Denote by $H=L^2(0, \pi)$ with inner product $(\xi,\eta)=\int_0^\pi \xi(x)\eta(x)dx$, norm $\|\xi\|_H=[\int_0^\pi \xi^2(x)dx]^{1/2}$ for any $\xi, \eta\in H$ and $X=C([-r, 0],H)$ the continuous function from $[-r, 0]$ to $H$ endowed with the supreme norm $\|\phi\|_{X}=\sup_{\theta \in[-r, 0]}\|\phi(\theta)\|_H$ for any $\phi \in X$. In order to set the solution in the abstract semigroup framework, We define $A: H\rightarrow H$ by
 \begin{equation}\label{5.3}
\begin{aligned}
Ay=\ddot{y}
\end{aligned}
\end{equation}
with domain $\operatorname{Dom}\left(A\right)=\left\{y \in C^2([0, \pi]) ; y(0)=y(\pi)=0\right\}$, $L: X\rightarrow H$ by
 \begin{equation}\label{5.2}
\begin{aligned}
L\phi \triangleq -a\phi(0)-b\phi(-r)
\end{aligned}
\end{equation}
for any $\phi\in C$ and
$A_U: X\rightarrow X$ by
 \begin{equation}\label{5.4}
\begin{aligned}
& A_U\phi=A\phi(0)+L\phi\\
\end{aligned}
\end{equation}
for any $\phi\in X$.
It is well known that $A-aI$ generates an analytic compact semigroup $\{T(t)\}_{t \geq 0}$ on $H$ and  \cite{WJ} that $A_U$ generates a semigroup $\{U(t)\}_{t \geq 0}$. Moreover, we also assume that $f$ satisfies the Lipschitz condition $\mathbf{Hypothesis\  A2}$.

It follows from \cite{WJ} Theorem 2.6 that \eqref{5.1} admits a global solution $u^\phi(\cdot):[-r, \infty] \rightarrow H$ such that $u(t)^\phi=\phi(t)$ for $t\in [-r,0]$ and
 \begin{equation}\label{5.5}
\begin{aligned}
u^\phi(t)=T(t) \phi(0)+\int_0^t T(t-s)\left[L\left(u_s^\phi\right)+f\left(u_s^\phi\right)\right] d s,
\end{aligned}
\end{equation}
for $t>0$.

Define $\Phi: \mathbb{R}\times X\rightarrow X$ by $\Phi(t)\phi=u^\phi_t(\cdot)$, then it generates a infinite dimensional dynamical system due to the uniqueness of the solution. The existence of attractor for semilinear or nonlinear partial functional differential equations including \eqref{5.1} as special case have been reported in much literature. See, for instance, \cite{YY} tackle the autonomous case with a nondensely defined linear part. Apparently, \eqref{5.1} satisfies other assumptions in \cite{YY} and hence it follows from Proposition 3.1 and Theorem 3.1 in \cite{YY} that \eqref{5.1} admits a global attractor, implying the existence of positive invariant absorbing set which is stated in the following.
\begin{lem}\label{lem5.1}
Assume that $\mathbf{Hypothesis\  A2}$ holds. Then, for any $\phi \in X$, there exists a constant $\gamma>a$ such that the integral solution $u^\phi_t(\cdot)$ of Eq. \eqref{5.1} satisfies the following inequality
 \begin{equation}\label{5.6}
\begin{aligned}
\left\|x_t\right\| \leq \frac{c_1 \mathrm{e}^{\gamma r}}{a-L_f \mathrm{e}^{\gamma r}}+\mathrm{e}^{\gamma r}\left(\|\phi\|-\frac{c_1}{a-L_f \mathrm{e}^{\gamma r}}\right) \mathrm{e}^{\left( L_f e^{\gamma r}-a\right) t}, \quad t \geq 0,
\end{aligned}
\end{equation}
where $c_1=\|f(\mathbf{0})\|$, $a \neq L_f e ^{\gamma r}$. If $a>L_f e^{\gamma r}$, then Eq. \eqref{5.1} has a nonempty global attractor $\mathcal{B}$.
\end{lem}

We now construct exponential attractor based on the global attractor $\mathcal{B}$ obtained in Lemma \ref{lem5.1}. Since $\mathcal{B}$ is positive invariant and absorbing, that is $\left(\mathcal{H}_1\right)$ and $\left(\mathcal{H}_2\right)$ are satisfied. We only need to estimate the dimensions of the global attractor botained in Lemma \ref{lem5.1}. We first introduce the following state decompose results of the linear part $A_U$ of \eqref{5.4} established in \cite{WJ}. It follows from  \cite{WJ} that the characteristic values of the linear part  $A_U$ are the roots of the following characteristic equation
 \begin{equation}\label{5.7}
\begin{aligned}
\Delta(\lambda)\xi=\left[A-\left(\lambda+a+b e^{-\lambda r}\right) I\right] \xi=0.
\end{aligned}
\end{equation}
Since $A_U$ is compact, it follows from Theorem 1.2 (i) in \cite{WJ} that the spectrum of $A_U$  are point spectra, which we denote by $\varrho_1>\varrho_2>\cdots$ with multiplicity $n_1, n_2,\cdots$, where $\varrho_1$ is defined as
 \begin{equation}\label{5.8}
\begin{aligned}
\varrho_1=\max \left\{\operatorname{Re} \lambda: \Delta(\lambda)\xi=\left[A-\left(\lambda+a+b e^{-\lambda r}\right) I\right] \xi=0\right\}.
\end{aligned}
\end{equation}
In the following, we always assume that $b-a<1$ and it follows from Lemma 1.13 on P73 in \cite{WJ} that if $a>0$, $b>0$ and $b-a<1$, then $\varrho_1<0$. For any given $\varrho_m<0$, $m\geq 1$, there is a
 \begin{equation}\label{5.8b}
\begin{aligned}
k_m=n_1+n_2+\cdots+n_m
\end{aligned}
\end{equation}
dimensional  subspace $X^U_{k_m}$ such that
$$X=X^U_{k_m} \bigoplus X^S_{k_m}$$
is the decomposition of $X$ by $\varrho_m$. Let $P_{k_m}$ and $Q_{k_m}$ be the projection of $X$ onto $X^U_{k_m}$ and $ X^S_{k_m}$ respectively, that is $X^U_{k_m}=P_{k_m}X$, $X^S_{k_m}=(I-P_{k_m})X=Q_{k_m}X$. It follows from the definition $P_{k_m}$ and $Q_{k_m}$ that
\begin{equation}\label{5.8a}
\begin{aligned}
\left\|U(t)Q_{k_m} x\right\| & \leq K e^{\varrho_m t}\|x\|, & & t \geq0,
\end{aligned}
\end{equation}
where $K$ and $\gamma$ are some positive constants.

By the same techniques as Theorem \ref{thm3.1}, we can obtain the following the squeezing property.
\begin{thm}\label{thm5.1}Let $P$ be  the finite dimension projection $P_{k_m}$ be defined by \eqref{5.8a}, $\varrho_{1}, \varrho_{m},\gamma$ and $K$ being defined in \eqref{5.8} and \eqref{5.8a} respectively, then we have
 \begin{equation}\label{5.12}
\left\|P \Phi(t)\varphi-P \Phi(t)\psi\right\| \leq 2e^{(L_f+\varrho_1)t}\left\|\varphi-\psi\right\|
\end{equation}
and
 \begin{equation}\label{5.13}
\begin{gathered}
\left\|(I-P) \Phi(t)\varphi-(I-P) \Phi(t)\psi\right\|\leq (Ke^{\varrho_m t}+\frac{KL_f }{\varrho_1+L_f-\varrho_m} e^{(L_f+\varrho_1)t})\left\|\varphi-\psi\right\|
\end{gathered}
\end{equation}
for any $t\geq 0$ and $\varphi, \psi\in\mathcal{M}$.
\end{thm}

By Theorem \ref{thm2.1}, we have the following results about the existence of exponential attractors of Eq. \eqref{5.1}.

\begin{thm}\label{thm5.3}
Let $k_m, \varrho_{1}, \varrho_{m},\gamma$ and $K$ be  defined in \eqref{5.8}, \eqref{5.8b} and \eqref{5.8a} respectively,  $P$ be  the finite dimension projection $P_{k_m}$ be defined by \eqref{5.8a}. Assume that the conditions of Lemma \ref{lem5.1} are satisfied. Moreover, assume   there exists $\alpha>0$ such that $\zeta:=\alpha e^{(L_f+\varrho_1)}+Ke^{\varrho_m}+\frac{KL_f }{\varrho_1+L_f-\varrho_m}e^{(L_f+\varrho_1)}<1$, , then, \eqref{5.1} admits an exponential  attractor $\mathcal{M}$ of which the fractal dimension has an upper bound
 \begin{equation}\label{3.34f}
\operatorname{dim}_f \mathcal{M}\leq \frac{k_m[\ln k_m +\ln(2+\frac{2}{\alpha })]}{-\ln \zeta}<\infty.
\end{equation}
\end{thm}

Similar to Corollary \ref{cor3.1}, we have the following corollary about the existence of exponential attractor $\mathcal{M}$ in the case $\varrho_m=\varrho_1$.

\begin{cor}\label{cor5.1}
 Let $\varrho_{1}$ and $K$ be  defined in \eqref{5.8}, \eqref{5.8b} and \eqref{5.8a} respectively and   $P$ be  the finite dimension projection $P_{k_1}$ be defined by \eqref{5.8a}. Assume that the conditions of Lemma \ref{lem5.1} are satisfied $a,b,r$ are appropriately chosen such that  $k_1=1$. Moreover, assume  there exists $\alpha>0$ such that $\alpha e^{(L_f+\varrho_1)}+Ke^{\varrho_1}+K e^{(L_f+\varrho_1)}<1$,
, then, \eqref{5.1} admits an exponential  attractor $\mathcal{M}$ of which the fractal dimension has an upper bound
 \begin{equation}\label{3.34g}
\operatorname{dim}_f \mathcal{M}\leq \frac{\ln(2+\frac{2}{\alpha })}{-\ln [(\alpha+K) e^{(L_f+\varrho_1)}+Ke^{\varrho_1}]}<\infty.
\end{equation}
\end{cor}

\section{Conclusions}
In this paper, we established a new  framework to construct exponential attractors for infinite dimensional dynamical systems in Banach spaces with explicit fractal dimension. The constructed  exponential attractors possess explicit  fractal dimensions which do not depend on the entropy number but only depend on inner characteristics of the studied equation. The method is especially effective for functional differential equations in Banach spaces for which sate decomposition of the linear part can be adopted to prove squeezing property. Nevertheless, we only consider the autonomous case. For  nonautonomous functional differential equations that generate evolution process, especially the case when the nonlinear term $f$ depends on time, i.e. $f=f(t,u_t)$, the semigroup approach adopted in this paper may be ineffective and the problem will be much more difficult. This problem will be tackled in an upcoming paper.

Generally, random effect are omnipresent in mathematical modelings.Therefore, one anther question is, whether there are exponential attractors  with explicit fractal dimension for partial functional differential equations perturbed by random effect, i.e. the stochastic partial functional differential equations(SPFDEs)? Indeed, even under what conditions do SPFDEs generate random dynamical systems have not been perfectly tackled needless to say the state decomposition and exponential dichotomy. This problem   also deserves much effort in the future.

\noindent{\bf Acknowledgement.}
This work was jointly supported by China Postdoctoral Science Foundation (2019TQ0089), China Scholarship Council(202008430247). \\
The research of T. Caraballo has been partially supported by Spanish Ministerio de Ciencia e
Innovaci\'{o}n (MCI), Agencia Estatal de Investigaci\'{o}n (AEI), Fondo Europeo de
Desarrollo Regional (FEDER) under the project PID2021-122991NB-C21 and the Junta de Andaluc\'{i}a
and FEDER under the project P18-FR-4509.\\
This work was completed when Wenjie Hu was visiting the Universidad de Sevilla as a visiting scholar, and he would like to thank the staff in the Facultad de Matem\'{a}ticas  for their hospitality and thank the university for its excellent facilities and support during his stay.

\end{document}